\documentclass[11pt]{amsart}
\usepackage{txfonts}
\usepackage{amscd,amssymb,amsmath,latexsym,verbatim} 
\usepackage[all]{xy}
\usepackage{breakurl}
\usepackage{multirow}
\usepackage{enumitem}
\usepackage{bbm}
\usepackage{nicefrac}
\usepackage{geometry}

\usepackage{import}
\usepackage{tikz}
\usepackage{caption}
\usepackage{subcaption}

\usepackage{color}
\usepackage[
            breaklinks=true, 
            colorlinks=true,
            linkcolor=red,
            urlcolor=blue,
            citecolor=dkgreen]{hyperref}

\usepackage{fancyhdr} 
\usepackage{titlesec}

\titleformat{\section}[block]{\normalfont\large\bfseries}
{\thesection.}{0.8em}{}

\titleformat{\subsection}[runin]{\normalfont\bfseries}
{\thesubsection.}{0.5em}{}[.]

\titleformat{\subsubsection}[runin]{\normalfont\normalsize\itshape}
{\thesubsubsection.}{0.5em}{}[.]

\usetikzlibrary{svg.path}
\definecolor{gG}{RGB}{ 60, 186,  84}
\definecolor{gY}{RGB}{244, 194,  13}
\definecolor{gB}{RGB}{ 72, 133, 237}
\definecolor{gR}{RGB}{219,  50,  54}

\newlength\htG
\protected\def\google{\settoheight{\htG}{G}%
  \begin{tikzpicture}[yscale=-1,scale=(\htG/240pt),baseline=(baseline)]
    \fill[fill=gG] svg {m797.49 249.7h35.975v-240.75h-35.975z};
    \coordinate (baseline) at (current bounding box.south);
    \fill[fill=gB] svg {m246.11 116.18h-116.57v34.591h82.673c-4.0842 48.506-44.44 69.192-82.533 69.192-48.736 0-91.264-38.346-91.264-92.092 0-52.357 40.54-92.679 91.371-92.679 39.217 0 62.326 25 62.326 25l24.22-25.081s-31.087-34.608-87.784-34.608c-72.197-0.001-128.05 60.933-128.05 126.75 0 64.493 52.539 127.38 129.89 127.38 68.031 0 117.83-46.604 117.83-115.52 0-14.539-2.1109-22.942-2.1109-22.942z};
    \fill[fill=gR] svg {m341.6 91.129c-47.832 0-82.111 37.395-82.111 81.008 0 44.258 33.249 82.348 82.673 82.348 44.742 0 81.397-34.197 81.397-81.397 0-54.098-42.638-81.959-81.959-81.959zm0.47563 32.083c23.522 0 45.812 19.017 45.812 49.66 0 29.993-22.195 49.552-45.92 49.552-26.068 0-46.633-20.878-46.633-49.79 0-28.292 20.31-49.422 46.741-49.422z};
    \fill[fill=gY] svg {m520.18 91.129c-47.832 0-82.111 37.395-82.111 81.008 0 44.258 33.249 82.348 82.673 82.348 44.742 0 81.397-34.197 81.397-81.397 0-54.098-42.638-81.959-81.959-81.959zm0.47562 32.083c23.522 0 45.812 19.017 45.812 49.66 0 29.993-22.195 49.552-45.92 49.552-26.068 0-46.633-20.878-46.633-49.79 0-28.292 20.31-49.422 46.741-49.422z};
    \fill[fill=gB] svg {m695.34 91.215c-43.904 0-78.414 38.453-78.414 81.613 0 49.163 40.009 81.765 77.657 81.765 23.279 0 35.657-9.2405 44.796-19.847v16.106c0 28.18-17.11 45.055-42.936 45.055-24.949 0-37.463-18.551-41.812-29.078l-31.391 13.123c11.136 23.547 33.554 48.103 73.463 48.103 43.652 0 76.922-27.495 76.922-85.159v-146.77h-34.245v13.836c-10.53-11.347-24.93-18.745-44.04-18.745zm3.178 32.018c21.525 0 43.628 18.38 43.628 49.768 0 31.904-22.056 49.487-44.104 49.487-23.406 0-45.185-19.005-45.185-49.184 0-31.358 22.619-50.071 45.66-50.071z};
    \fill[fill=gR] svg {m925.89 91.02c-41.414 0-76.187 32.95-76.187 81.57 0 51.447 38.759 81.959 80.165 81.959 34.558 0 55.768-18.906 68.426-35.845l-28.235-18.787c-7.3268 11.371-19.576 22.484-40.018 22.484-22.962 0-33.52-12.574-40.061-24.754l109.52-45.444-5.6859-13.318c-10.58-26.08-35.26-47.86-67.92-47.86zm1.4268 31.413c14.923 0 25.663 7.9342 30.224 17.447l-73.139 30.57c-3.1532-23.667 19.269-48.017 42.915-48.017z};
  \end{tikzpicture}%
}

\newcounter{first}
{\end{list}}

\definecolor{dkgreen}{rgb}{0.1,0.4,0.0}
\definecolor{dkblue}{rgb}{0,0.1,0.8}
\definecolor{dkred}{rgb}{1,0,0}

\def\define{\ensuremath{\overset{\operatorname{\scriptscriptstyle def}}=}}

{ 
 
 \begin{enumerate}}{\end{enumerate}}

\theoremstyle{plain}
\newtheorem{theorem}[subsection]{Theorem}
\newtheorem{TheoM}{Theorem}

\newtheorem*{ShapleyThm*}{Shapley's Theorem \cite{Weber-robabilistic-values-for-games}}

\newtheorem{lemma}[subsection]{Lemma}
\newtheorem{proposition}[subsection]{Proposition}

\theoremstyle{definition}
\newtheorem{definition}[subsection]{Definition}

\theoremstyle{remark}
\newtheorem{remark}[subsection]{Remark}
\newtheorem{example}[subsection]{Example}

\newtheorem*{conjecture*}{Conjecture}

\newcommand{\doi}[1]
{\texttt{\href{http://dx.doi.org/#1}{\nolinkurl{doi:#1}}}}
\newcommand{\web}[1]
{\texttt{\href{#1}{\nolinkurl{#1}}}}

\title{Probabilistic values for simplicial complexes}
\author{Ivan Martino}
\date{\today}

\begin{document}

\def\rank#1{\ensuremath{\operatorname{rk} #1}}
\def\face{\ensuremath{\operatorname{\mathcal F} (\Delta)}}

\def\FacetsD{\ensuremath{\operatorname{Ft} \Delta}}	
\def\Facets#1{\ensuremath{\operatorname{Ft} #1}}
\def\Facet#1#2{\ensuremath{\operatorname{Ft}_{#2} #1}}

\def\closure#1#2{\ensuremath{\operatorname{cl}_{#1} (#2)}}
\def\closureDelta#1{\ensuremath{\operatorname{cl}_{\Delta} (#1)}}

\def\Star#1#2{\ensuremath{\operatorname{St}_{#2} #1 }}
\def\StarNoD#1{\ensuremath{\operatorname{St} #1 }}
\def\StarD#1{\ensuremath{\operatorname{St}_{\Delta} #1 }}

\def\Link#1#2{\ensuremath{\operatorname{Lk}_{#2} #1 }}
\def\LinkNoD#1{\ensuremath{\operatorname{Lk} #1 }}
\def\LinkD#1{\ensuremath{\operatorname{Lk}_{\Delta} #1 }}

\def\ffD{\ensuremath{\operatorname{\textbf{f}}(\Delta)}}
\def\ff#1{\ensuremath{\operatorname{\textbf{f}}(#1)}}
\def\ffi#1#2{\ensuremath{\operatorname{f}_{#1}(#2)}}

\def\charFun{\ensuremath{\operatorname{\mathbb{R}_{\Delta}} }}

\def\vtot{\ensuremath{\operatorname{v}_{\Delta}}}
\def\vtotof#1{\ensuremath{\operatorname{v}_{#1}}}

\def\Coalition#1#2{\ensuremath{\operatorname{Coalitions}_#2 #1 }}
\def\subgroupD{\ensuremath{\pi(\Delta)}}

\begin{abstract}
In this manuscript, we define and study probabilistic values for cooperative games on simplicial complexes.
Inspired by the work of Weber \cite{Weber-robabilistic-values-for-games}, we establish the new theory step by step, following the classical axiomatization, i.e. using the \emph{linearity} axiom, the \emph{dummy} axiom, etc.

Furthermore, we define Shapley values on simplicial complexes generalizing the classical notion in literature. Remarkably, the traditional axiomatization of Shapley values can be extended to this general setting for a rather interesting class of complexes that generalize, in certain instances, the notions of \emph{vertex-transitive} graphs \cite{MR3848666, MR0460172, MR941942} and \emph{vertex-homogeneous} simplicial complexes \cite{MR1373690, MR779890, MR1871694, MR1809429, MR1710494}. These combinatorial objects are very popular in the literature because of the study of Evasiveness Conjecture in Complexity Theory \cite{MR3848666, MR0460172, MR941942, MR1871694, MR1809429}. 
\end{abstract}

\maketitle

A cooperative game on $[n]\define \{1, \dots, n\}$ is a \emph{characteristic} function defined on all subsets of $[n]$, $v: 2^n \rightarrow \mathbb{R}$ with the restriction $v(\emptyset)=0$.
In other words, $n$ is the number of \emph{players} and $v(T)$ provides the \emph{worth} of every \emph{coalition} $T\subseteq [n]$.
It is a very interesting problem to determine an equitable distributions of the payoff of the grand coalition $v([n])$, see for instances \cite{Shapley-a-value, Shapley-core-convex, Weber-robabilistic-values-for-games}.

In the last decade, several research articles have considered cooperative games on different combinatorial structures \cite{Shapley-matroids-static, Shapley-matroids-dynamic, MR3886659, MR2847360, MR2825616, MR1436577, MR1707975}.
%
Inspired by these, the author has defined and studied cooperative games on \emph{simplicial complexes} \cite{Martino-cooperative} and he has developed the concept of efficiency for this new class of games \cite{Martino-Efficiency}.

In this manuscript, we define probabilistic values for cooperative games on simplicial complexes. Influenced by the work of Weber \cite{Weber-robabilistic-values-for-games}, we establish the new theory step by step, following the classical axiomatization, i.e. using the \emph{linearity} axiom, the \emph{dummy} axiom, etc.
Further, we define Shapley values on simplicial complexes generalizing the classical notion in literature. 

Remarkably, Shapley values can be only characterized for a rather interesting class of complexes that generalize, in specific instances, the notions of \emph{vertex-transitive} graphs \cite{MR3848666, MR0460172, MR941942} and \emph{vertex-homogeneous} simplicial complexes \cite{MR1373690, MR779890, MR1871694, MR1809429, MR1710494}. These combinatorial objects are very important in the literature because of the study of Evasiveness Conjecture in Complexity Theory \cite{MR3848666, MR0460172, MR941942, MR1871694, MR1809429}. 

Before proceeding with a detailed presentation of the new results, we are going to introduce cooperative games on simplicial complexes and, right after, we will provide several motivation for this work.

\subsection*{Games on simplicial complexes}
Playing the game in the traditional setting, we assume that each player may join every coalition. The main motivation for the generalization is that this may not be allowed: there could be certain coalitions forbidden for very specific reasons. Once this is on the ground, one could also argue that if the coalition $T$ is not allowed, then every other coalition containing it will not be accepted too. 

Simplicial complexes are combinatorial and topological objects that serve as coalition set for these games because they respect the above consideration. Specifically, a simplicial complex is a family $\Delta$ of subsets of $[n]$ closed under inclusion, that is if $X\in \Delta$, then every subset $Y\subseteq X$ belongs also to $\Delta$. 
There are several familiar examples: graphs, the boundary of simplicial polytopes, and the full power set $2^{n}$ are simplicial complex. In the latter case, $\Delta=2^{n}$ is called a $(n-1)$-dimensional simplex.

\noindent
A cooperative game on $\Delta$ is defined by a characteristic function $v$:
\[
	v: \Delta \rightarrow \mathbb{R}
\]
with the usual constrain that $v(\emptyset)=0$.
Employing this notation, the game $([n], v)$ can be seen as the cooperative game on the $(n-1)$-dimensional simplex $(2^n, v)$, with the same characteristic function $v$.
So these cooperative games are the ones where every coalition is allowed.

As in the traditional case, the individual function $\phi_i(v)$ measures the additional value that the playes $i$ provides to a feasible coalition during the cooperative game $(\Delta, v)$.
The focus of this work is the study of the individual values for the cooperative games where \emph{feasible} coalition $T$ are defined to be elements of $\Delta$.
%

\vspace{0.1cm}

It is worth to mention a few reasons why the generalization provided is relevant.

\subsection*{Probabilistic values for simplicial complexes}
Matroids are at the intersection of Algebra, Combinatorics, Geometry, and Topology. Since their introduction \cite{Whi35}, new variations have appeared in literature encoding different type of independence \cite{MR2989987, MR2989987, MR1164708, MR1433646,MR2844079, MR3944531, Martino2018, Borzi-Martino-D-matroids, MR3883211, Borzi-Martino-DVR-matroids}.

Several authors have already taken in consideration cooperative games on matroids
\cite{Shapley-matroids-static, Shapley-matroids-dynamic, MR3886659, MR2847360, MR2825616, MR1436577, MR1707975}. Matroids are simplicial complexes fulfilling the \emph{base exchange property}, see for instance \cite{Stanley2012b, Stanley1996a,  MR782306, Oxley}. As shown in Figure \ref{fig:totale}, not every pure simplicial complex is a matroid: Figure \ref{fig:figure-matroid-yes} and \ref{fig:figure-matroid-no} show two pure simplicial complexes and only the one in the left is a matroid.

The author has introduces cooperative games on simplicial complexes in \cite{Martino-cooperative} and studied quasi-probabilistic values introduced in \cite{Shapley-matroids-static}.
Inspired by \cite{Weber-robabilistic-values-for-games}, in this manuscript we develop the theory for probabilistic values and in Section \ref{sec:linearity} and Section \ref{sec:dummy} we extend the results of Weber \cite[Section 3 and 4]{Weber-robabilistic-values-for-games}, by borrowing (from Combinatorial Topology) the notion of link of a vertex.

\begin{figure}
\centering
\begin{subfigure}[b]{5cm}
\begin{tikzpicture}[scale=0.70]
   	\coordinate (F) at ( -2.0,  1.0);
  	\coordinate (S) at (-3.0, -2.0); 
	\coordinate (FB) at ( 0.0, -0.0); 
	\coordinate (TV) at (2.0, 1.0); 
	\coordinate (E) at (3.0, -2.0);
	 
   \node [left] at ( -2.0,  1.0) {1}; 
   \node [left] at (-3.0, -2.0) {2}; 
	\node [above] at ( 0.0, -0.0) {3}; 
   \node [right] at (2.0, 1.0) {4}; 
   \node [right] at (3.0, -2.0) {5};

   \filldraw [draw=black, fill=red!20, line width=1.5pt] (FB)--(TV)--(E) -- (FB);
   
   \filldraw [draw=black, fill=green!20, line width=1.5pt] (FB)--(S)--(E) -- (FB);

\filldraw [draw=black, fill=blue!20, line width=1.5pt] (FB)--(F)--(S) -- (FB); 
   
   
\end{tikzpicture}
        \subcaption{This simplicial complex is a matroid.}
        \label{fig:figure-matroid-yes}
    \end{subfigure}
~\hspace{0.5cm}
	\begin{subfigure}[b]{5cm}
\begin{tikzpicture}[scale=0.75]
   	\coordinate (F) at ( -2.0,  1.0);
  	\coordinate (S) at (-3.0, -2.0); 
	\coordinate (FB) at ( 0.0, -0.0); 
	\coordinate (TV) at (2.0, 1.0); 
	\coordinate (E) at (3.0, -2.0);
	 
   \node [left] at ( -2.0,  1.0) {1}; 
   \node [left] at (-3.0, -2.0) {2}; 
	\node [above] at ( 0.0, -0.0) {3}; 
   \node [right] at (2.0, 1.0) {4}; 
   \node [right] at (3.0, -2.0) {5};

   \filldraw [draw=black, fill=red!20, line width=1.5pt] (FB)--(TV)--(E) -- (FB);

   \filldraw [draw=black, fill=blue!20, line width=1.5pt] (FB)--(S)--(F) -- (FB);
   
\end{tikzpicture}
        \subcaption{This simplicial complex is not a matroid.}
        \label{fig:figure-matroid-no}
	\end{subfigure}

\caption{Two examples of pure simplicial complexes.}\label{fig:totale}
\end{figure}
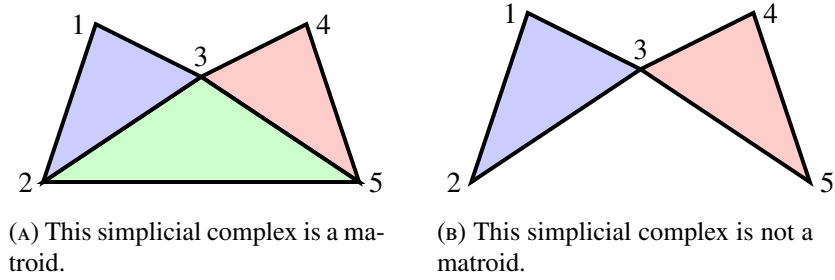

\subsection*{Shapley values}
As anticipated, one of our goal is to define Shapley values for simplicial complexes.
A first tentative for this can be track back to Bilbao, Driessen, Jim\'{e}nez Losada and Lebr\'{o}n \cite{Shapley-matroids-static} for cooperative game on matroids. They study individual values that can be written as weighted sum of classical Shapley values. The author has shown that this idea can be generalized to \emph{pure} simplicial complexes (i.e. every maximal element has the same cardinality) but not further, see Theorem 5.2 of \cite{Martino-cooperative}.

This restriction is kind of unwelcome, because Shapley values are tremendous objects for this theory. Thus, in this work, we provide a natural definition of Shapley values for every simplicial complex, including the not pure ones, see Section \ref{sec:Shapley}.
On the other hand, the elegant characterization, seen in Section 9 of \cite{Weber-robabilistic-values-for-games}, holds only for a specific class of complexes having strong symmetry. 

\subsection*{Players with different strategies}
The request of symmetry for $\Delta$ translates to players having the same strategy constrains.
The combinatorial and topological side of this phenomenon will be exposed in the next paragraph, but here we highlight an observation relevant for the theory of these games.
%
In the traditional setting, players may have different \emph{point of view} and \emph{attitude to risk}. In this new setting, players might also have different \emph{strategy constrains}: since certain coalitions are not allowed, different players could have completely different strategy to the game. For instance, in the example in Figure \ref{fig:figure-matroid-no}, the player $3$ is the only one who may join two different coalitions of cardinality three, the largest coalition possible for this game.
%



\subsection*{The Shapley complex}
The Shapley values are characterized axiomatically only if the link of every vertex of the simplicial complex has the same $f$-vector, say $\textbf{s}$; we call this complexes, $\textbf{s}$-Shapley, see Section \ref{sec:Shapley}.
This definition generalizes, in specific instances, the notion of \emph{vertex-homogeneous} complexes, that is a simplicial complex $\Delta$ such that every pair of vertices $i, j\in \Delta$, there exists an automorphism $g\in Aut(\Delta)$ such that $g(i) = j$, see more in Section 6 of \cite{MR2872540}.

We are not going to use this notion, but we want to remark that vertex-homogeneous simplicial complexes are extremely useful in the study of non-evasiveness, used in Combinatorics and combinatorial Geometry, and motivated by questions in Graph Theory \cite{MR1373690, MR779890, MR1871694, MR1809429, MR1710494}.
A very nice and well written introduction about simplicial complexes and the evasiveness conjecture can be found here \cite{MR3848666}.

\subsection*{Machine-Learning prediction models and Multi-Touching Attribution}
A ground-breaking application of the Shapley value methodology is in the interpretability of the Machine-Learning prediction models \cite{NIPS2017_7062, Ribeiro, Erik-Igor, Shrikumar, 7546525, OnPixel, Lipovetsky}.

Last but not the least, another very important application is the multi-touch attribution system in Marketing, for instance offered in marketing platform of \google. 
A prototype example of the multi-touch attribution system in Marketing on simplicial complexes is provided in the Introduction of \cite{Martino-cooperative}.

\section*{Presentation of the new results}

Our first observation is that the set of coalition, $\Coalition{i}{\Delta}$, the player $i$ can join is actually a simplicial complex it-self. Indeed, we show in Proposition \ref{prop:coalition-link} that
\begin{equation*}
		\Coalition{i}{\Delta}=\Link{i}{\Delta},
\end{equation*}	
the link of the vertex $i$ in $\Delta$, see Definition \ref{def:link}.

Then, we start our axiomatization of the individual values. 
Let us denote by $\charFun$ the vector space of cooperative games on the simplicial complex $\Delta$, i.e. the set of all characteristic functions under the constrain $v(\emptyset)=0$.
If $c$ is a real number, then the re-scaled game $(n, c\cdot v)$, given by the characteristic function $(c\cdot v)(T)=cv(T)$ for every $T\in \Delta$, should provide a re-scaled (by the same constant) player assessments, and so a re-scaled individual values $\phi_i(c\cdot v)=c\phi_i(v)$. 

\begin{description}
	\item [We assume] the individual value $\phi_i$ is defined in a cone $\mathfrak{I}$ of cooperative games in $\charFun$.
\end{description}

\noindent
The first requirement is that the $\phi_i$ are linear functions.

\begin{description}
	\item [Linearity Axiom] 
	The individual value $\phi_i(v)$ of the player $i$ for the cooperative game $(\Delta, v)$ is linear on $\mathfrak{I}$.
\end{description}

\noindent
It seems reasonable to think that the individual value $\phi_i(v)$ of the player $i$ only depends on the payoff  $v(T)$ and $v(T\cup i)$ for every set $T$ in the coalition set $\Coalition{i}{\Delta}$. 
This is precisely the content of our first statement.
To appreciate the results, let us remember that the \emph{star} of a player $i$ is, as a set, $\Star{i}{\Delta}\define \{T, T\cup i: T\in \Coalition{i}{\Delta}\}$; this will be carefully defined in Definition \ref{def:star}.

\begin{TheoM}
	%
	Assume that the cone of cooperative games $\mathfrak{I}$ contains the carrier games $\mathcal{C}$ and $\hat{\mathcal{C}}$ for every $T\in\Delta$ but $T\cup i\notin \Delta$.
	%
	If the individual value $\phi_i$ fulfills the linearity axiom, 
	then, $\phi_i(v)$ only depends on the value of $v$ on the Star of $i$ in $\Delta$, $\Star{i}{\Delta}$, that is
	\begin{equation*}
	\phi_i(v)=\sum_{T\in \Star{i}{\Delta}} a_T v(T).
\end{equation*}
The set of coefficients $\{a_T\}$ is unique.
\end{TheoM}

Assume that a player does contribute the same specific amount $v(i)$, the number of payoff s/he can reach by himself, to every coalition: such player is traditionally called \emph{dummy}, because s/he has no strategy, not matter what coalition s/he joins.
Then, the next axiom imposes that her/his retribution should be precisely $v(i)$.

\begin{description}
	\item [Dummy Player Axiom] 
If a player $i$ is dummy for the cooperative game $(\Delta, v)$, then its individual value $\phi_i(v)$ is $v({i})$.
\end{description}

\noindent
This axiom gives the very traditional formula expressing $\phi_i(v)$ in terms of the marginal contributions $\left(v(T\cup i) - v(T)\right)$:

\begin{TheoM}
	Assume that the cone of cooperative games $\mathfrak{I}$ contains the carrier games $\mathcal{C}$ for $\Link{i}{\Delta}$.
	If the individual value $\phi_i$ can be written as in equation \eqref{eq:linearity-axiom} and satisfies the dummy player axiom, then
	there exist real numbers 
	$\{p_T\}_{T\in \Link{i}{S}}$ 
	such that
	$$\sum_{T\in \Link{i}{\Delta}} p_T=1$$
	and for every $v$ in $\mathcal{I}$
	\begin{equation}\label{eq:012-phi-p_T-introduzione}\tag{$\square$}
		\phi_i(v)=\sum_{T\in \Link{i}{\Delta}}p_T \left(v(T\cup i) - v(T)\right).
	\end{equation}	
\end{TheoM}

Next, we want to deal with monotone functions: $v:\Delta \rightarrow \mathbb{R}$ such that for every $S\subseteq T\in \Delta$, then $v(S)\leq v(T)$. 
If $v$ is a monotone function, then every marginal contribution is positive, because $v(T\cup i)\geq v(T)$. Hence, every player will require a positive retribution. 

\begin{description}
	\item [Monotonicity Axiom] 
	If $v$ is a monotone game, then $\phi_i(v)\geq 0$ for every player $i$.
\end{description}

Using Theorem \ref{thm:la+da} or, more generically, if the individual value can be written as in equation \eqref{eq:012-phi-p_T-introduzione}, one can immediately show that the real constants $\{p_T\}_{T\in \Link{i}{\Delta}}$ are actually a probability distribution on the elements of $\Link{i}{\Delta}$. 

\begin{TheoM}
	Assume that the cone of cooperative games $\mathfrak{I}$ contains the carrier games $\mathcal{C}$ and $\hat{\mathcal{C}}$ for $\Star{i}{\Delta}$.
	We assume that the individual value $\phi_i$ can written as in equation \eqref{eq:012-phi-p_T-introduzione}, that $\sum_{T\in \Link{i}{\Delta}} p_T=1$, and that it
	fulfills 
	the monotonicity axiom.
	Then, the constants $\{p_T\}_{T\in \Link{i}{\Delta}}$ are a probability distribution.
\end{TheoM}

With these results in mind, a \textbf{probabilistic value} defined on the cone $\mathfrak{I}$ of cooperative games on a simplicial complex $\Delta$ is the linear function $\phi_i$ written in the form $\eqref{eq:012-phi-p_T-introduzione}$ such that $\{p_T\}_{T\in \Link{i}{\Delta}}$ is a probability distribution on the elements of link of the vertex $i$.

In \cite{Martino-cooperative}, with a different probabilistic flavor, the author has shown similar results but distinct from the one in Theorem \ref{thm:linearity-axiom-step}, \ref{thm:la+da} and \ref{thm:la+da+ma}. Aside that the style of proof differs and that the tools used are diverse (for instance, in \cite{Martino-cooperative} the author borrow the $\lambda_i$-dummy axiom from \cite{Shapley-matroids-static}), the main goal of \cite{Martino-cooperative} is studying when a \emph{quasi}-probabilistic value can be written as sum of traditional Shapley values.

\vspace{0.2cm}
Up to this achievement, everything somehow follows the path shown in \cite{Weber-robabilistic-values-for-games}. 
%
%
The two theories parts introducing the symmetric axiom. 
We start by recalling that a \textbf{group value} $\operatorname{\phi}$ for the cooperative game $(\Delta, v)$ is the vector $(\phi_1(v), \dots, \phi_n(v))$, collecting all individual player values.

We denote by $\operatorname{Symm}(\Delta)$ the largest subgroup of the symmetric group $S_n$ such that if $T\in \Delta$ and if $g\in \operatorname{Symm}(\Delta)$, then $\pi T\in \Delta$.
Moreover, $\pi \cdot v$ stands for the permuted characteristic function; this is $(\pi \cdot v): \Delta \rightarrow \mathbb{R}$ such that
\[
	\pi \cdot v(T)=	v(\pi T).
\]

\begin{description}
	\item [Symmetry Axiom] 
	Assume that for every permutation $\pi$ in $\operatorname{Symm}(\Delta)$ we have that $\pi \cdot v$ belongs to $\mathfrak{I}$.
	Let $\phi$ be a group value for the cooperative game $(\Delta, v)$.
	Then, $\phi_i(v)=\phi_{\pi i}(\pi \cdot v)$, for every $\pi$ in $\operatorname{Symm}(\Delta)$.
\end{description}

\noindent
Some of the symmetries in $S_{n}$ have a key role in the theory. We define $\pi_{L,T}^i$ to the permutation that switches two set $L$ and $T$ of the same cardinality in $\Link{i}{\Delta}$ and fixes $i$. We denote by $\pi_{i,j}$ the simple transposition $(i,j)$.
In Definition \ref{def:subgroup}, we denote by $\subgroupD$ the subgroup of $S_n$ generated by the permutation $\pi_{L,T}^i$ for all $i$ and for all $L,T\in\Link{i}{\Delta}$ with $|L|=|T|$ and generated by every transposition $\pi_{i,j}$ provided that $\Link{i}{\Delta}\cap \Link{j}{\Delta}\neq \emptyset$.

In our first analysis in Theorem \ref{thm:s.a.}, we are able to show that the symmetric axiom implies that the probability $p_T$ for every non-trivial coalition $T$ actually depends only on the cardinality of $T$. 
On the other hand, in Section \ref{sec:restriction-for-equidim-links} and Section \ref{sec:bit-system-f-vector}, we explain that if we assume the symmetric axiom, then without loss of generality we may work with simplicial complexes with pure links $\Link{i}{\Delta}$ or rank $r-1$. We shortly call these complexes as \emph{simplicial complex with pure links}, see Definition \ref{def:pure-links-simplicial-complex}.

To present the next result it is necessary the notion of $f$-vector of a simplicial complex: this is the vector of integers $\ffD=(\ffi{-1}{\Delta}, \dots, \ffi{r-1}{\Delta})$ that encodes the number of elements of a specific cardinality, that is
\begin{equation*}
	\ffi{i}{\Delta}\define|\{S\in \Delta, \#S=i+1\}|.
\end{equation*}
For instance, $\ffi{-1}{\Delta}=1$, because the empty set always belongs to any non-empty simplicial complex and, if the complex is pure of rank $r$, then \ffi{r-1}{\Delta} gives the number of facets of $\Delta$, that is the cardinality of $\Facets{\Delta}$, see Section \ref{sec:preliminaries}.
%

\begin{TheoM}
	Let $\Delta$ be a simplicial complex of rank $r$ \textbf{with pure links} and let $\mathfrak{I}$ be a cone of cooperative games defined on $\Delta$, containing the carrier games $\mathcal{C}$ and $\hat{\mathcal{C}}$ for the star $\Star{i}{\Delta}$ for every vertex $i$.
    Assume that if $v\in\mathfrak{I}$, then the permuted game $\pi \cdot v$ is also in $\mathfrak{I}$ for every $\pi$ in $\operatorname{Symm}(\Delta)$.
    Let $\phi$ be a group value such that for each $i\in [n]$ and for each $v\in\mathfrak{I}$, we can write:
    \[
		\phi_i(v)=\sum_{T\in \Link{i}{\Delta}}p_T^i \left(v(T\cup i) - v(T)\right).
    \]

    If $\subgroupD \subseteq \operatorname{Symm}(\Delta)$ and if $\phi$ satisfies the symmetric axiom, then there are real constants $p_t$ with $\mathbf{t \neq 0}, \dots, \operatorname{rank}\Delta - 1$ such that for every players $i$ and every \emph{non-trivial} $T\in \Link{i}{\Delta}$ one has $p_T^i=p_{|T|}$.

	Moreover, there exists a common value $p_0$ for the probabilities $p_{\emptyset}^i$ if $\textbf{p}=(p_0, p_1, \dots, p_{r-1})$ is the solution of the following system of $n$ linear equations:
	{\em
	\begin{equation*}
	    \left\{ \begin{array}{c}
	         \ff{\Link{1}{\Delta}}\cdot \textbf{p}=1,\\
	         \ff{\Link{2}{\Delta}}\cdot \textbf{p}=1,\\
	          \vdots \\
	         \ff{\Link{n}{\Delta}}\cdot \textbf{p}=1,
	    \end{array} \right.
	\end{equation*}}
	where each equation has form $\sum_{k=0}^{r-1}\ffi{k-1}{\Link{i}{\Delta}} p_{k} =1$.
\end{TheoM}

\noindent
In Theorem \ref{thm:at-least-a-solution-shapley-complex}, we show that if $\Delta$ is a Shapley simplicial complex, then the previous linear system has at least a solution.

\vspace{0.2cm}
The Shapley values are probabilistic values for the cooperative game $(2^{[n]}, v)$ where the $p_T$'s are set following that every player joins coalitions of different sizes with the same probability and all coalitions of the same size are equally likely.
Using this recipe, we define the Shapley values for every simplicial complex $\Delta$. We note that the coalition may have cardinality among $0\leq k\leq r_i\define \operatorname{rank}\Link{i}{\Delta}$ and the number of coalitions of a specific cardinality $k$ is encoded by the $(k-1)-th$ component of the $f$-vector of $\Link{i}{\Delta}$:
%
\begin{equation}\label{eq:Shapley-def-simplicial-complex-introduction}\tag{$\star$}
	\operatorname{Shapley}^{\Delta}_i(v)=\frac{1}{r_i+1} \sum_{T\in \Link{i}{\Delta}}  \frac{1}{\ffi{|T|-1}{\Link{i}{\Delta}}} (v(T\cup i)-v(T)).
\end{equation}

\noindent
In the traditional case, $\Delta$ is a simplex,  $f_{|T|-1}=\frac{|T|!(n-|T|-1)!}{(n-1)!}$, and $r_i+1=n$; so:
\[
	\operatorname{Shapley}_i(v)=\sum_{T\subseteq [n]\setminus i} \frac{1}{n}\frac{|T|!(n-|T|-1)!}{(n-1)!} (v(T\cup i)-v(T)).
\]
With this specifics ($\Delta=2^n$), these individual values are characterized by the following theorem:

\begin{ShapleyThm*}
	Let $\Delta=2^{[n]}$ and let $\mathfrak{I}$ be a cone of cooperative games containing the carrier games $\mathcal{C}$ and $\hat{\mathcal{C}}$ for $N\setminus i$.
	Assume that if $v\in\mathfrak{I}$, then the permuted game $\pi \cdot v$ is also in $\mathfrak{I}$ for every permutation $\pi$ of $[n]$.

	Let $\phi$ be a group value. If each $\phi_i$ satisfies the linearity axiom, the dummy player axiom, 
	and if the symmetric axiom and the \textbf{efficiency axiom} hold for the group value $\phi$, then for every cooperative game in $\mathfrak{I}$ and every $i$ in $[n]$, 
	\[
		\phi_i(v)=\operatorname{Shapley}_i(v).
	\]
\end{ShapleyThm*}

As highlighted with bold text, we did not treat, yet, a piece of the axiomatization: the efficiency.  
The group value $\operatorname{\phi}=(\phi_1, \phi_2, \dots, \phi_n)$ assessment is optimistic (w. r. to $v$) if the sum of the payoff vector $\sum_i \phi_i(v)$ is greater than the $v([n])$, the worth of the grand coalition. If the contrary happens, then $\phi$ is pessimistic (w. r. to $v$).
The classical Shapley values are characterized by group values that are nor optimistic or pessimistic, that is $\sum_i \phi_i(v)=v([n])$. 

If certain coalitions are forbidden, so the grand coalition would be. Thus, it is necessary to study what could take the place of the the total number of payoff, that in the traditional case is simply to $v([n])$.
This study is the center of attention of the manuscript \cite{Martino-Efficiency}.

The proposed characterization of the Shapley values for cooperative games on simplicial complex used a different efficiency scenario. To present it, we need the following objects: we define a player $j$ to be an extension for the coalition $T$ if $j$ can join $T$ in the cooperative game:
\begin{equation*}
	\operatorname{Ext}(T)=\{j\in [n]\setminus T: T\cup j \in \Delta \};
\end{equation*}
The cardinality of this set is denoted by $\operatorname{ext}(T)=|\operatorname{Ext}(T)|$.

\begin{TheoM}[Generalized Shapley's Theorem]
	Let $\Delta$ be a simplicial complex of rank $r$ with pure links and let $\mathfrak{I}$ be a cone of cooperative games on $\Delta$ containing the carrier games $\mathcal{C}$ and $\hat{\mathcal{C}}$ for the star $\Star{i}{\Delta}$ for every vertex $i$.
Assume that if $v\in\mathfrak{I}$, then the permuted game $\pi \cdot v$ is also in $\mathfrak{I}$ for every permutation $\pi$ in $\operatorname{Symm}(\Delta)$.

Let $\phi$ be a group value and assume that each $\phi_i$ satisfies the linearity axiom, the dummy player axiom. 
Assume also that 
the group value fulfills the symmetric axiom with the further constrain that every player enters the games with probability $p_0$, that is there exists a solution $\textbf{p}=(p_0, p_1, \dots, p_{r-1})$ for the linear system $n$ equations:
{\em
	\begin{equation*}
        \left\{ \begin{array}{c}
	         \ff{\Link{1}{\Delta}}\cdot \textbf{p}=1,\\
	         \ff{\Link{2}{\Delta}}\cdot \textbf{p}=1,\\
	          \vdots \\
	         \ff{\Link{n}{\Delta}}\cdot \textbf{p}=1,
	    \end{array} \right.
	\end{equation*}}
	where each equation has the form in $\sum_{k=0}^{r-1}\ffi{k-1}{\Link{i}{\Delta}} p_{k} =1$.

Then, every individual value is the Shapley value in equation \eqref{eq:Shapley-def-simplicial-complex-introduction} if and only if $\Delta$ is $\textbf{s}$-Shapley with $\textbf{s}=(s_0,\dots, s_{r-1})$ and the group value satisfies the following efficiency scenario:
\begin{equation*}
		\sum_{i\in [n]}\phi_i(v)=\sum_{F\in \FacetsD} %
	\left( \frac{1}{s_{r-1}} \right)	v(F) 
	+\frac{1}{r}\sum_{T\in \Delta, |T|<r} 
	\left(\frac{|T|}{s_{|T|-1}} -\frac{\operatorname{ext}(T)}{s_{|T|}}\right) 
	v(T).
\end{equation*}	
\end{TheoM}

The efficiency constrain may seem artificial, but so it is any efficient condition, as Weber pointed out in Section 8 of \cite{Weber-robabilistic-values-for-games}.

Finally, in Theorem \ref{thm:decomposing}, we link this new definition with the development in \cite{Shapley-matroids-static, Shapley-matroids-dynamic, Martino-cooperative}, by finding the extra condition for writing $\operatorname{Shapley}^{\Delta}_i$ as sum of classical Shapley values.

\subsection*{Acknowledgments} 
	The author is currently supported by the Knut and Alice Wallenberg Foundation and by the Royal Swedish Academy of Science.

%
%
%


\setcounter{section}{0}
\setcounter{subsection}{0}
\setcounter{TheoM}{0}

\section{Basic Definitions}\label{sec:preliminaries}

For every positive integer $n$, we denote by $[n]\define\{1,\dots, n\}$. This is the set of verticies of the simplicial complex and, hence, the set of players of the cooperative game. 

\subsection{Simplicial Complexes}
A finite simplicial complex $\Delta$ over $n$ verticies is a family of subsets in $2^{[n]}$ such that if $S\in\Delta$, then very subset $T$ of $S$ will also belong to the family, $T\in \Delta$.
We say that a simplicial complex is non empty if $\Delta\neq \emptyset$ and, if this is the case $\emptyset\in \Delta$ because the empty set is always a subset for every $S\in \Delta$.

\noindent
The family $\Delta$ has natural a rank function $\rank$ given by the cardinality of its sets:
\[
	\begin{array}{cccc}
		\rank:&\Delta & \rightarrow & \mathbb{N}\\
		 &T	&\mapsto & |T|.
	\end{array}
\]
The elements of $\Delta$ that are maximal by inclusions are called \emph{facets}, the other ones are faces.
The maximal value of the rank function is the rank of the simplicial complex $\rank{\Delta}$. 
The set of facets of $\Delta$ is  $\FacetsD=\{F_1, \dots, F_k\}$ and for every $S\in \Delta$ let $\Facet{S}{\Delta}$ be the set of facets in $\Delta$ that contain $S$.

%

In this manuscript, we assume that $\Delta$ is always a finite non-empty simplicial complex over $n$ verticies of rank $r\define \rank{\Delta}$. 


If $S$ is an element in $\Delta$, then $\bar{S}\define 2^{S}$ is the $(|S|-1)$-dimensional simplex defined on the verticies of $S$. The next two definitions are very important for this work.

\begin{definition}\label{def:star}
The star of an element $S$ in $\Delta$ is the simplicial complex defined to be the collection of all subset in $\bar{T}$ with $T$ being in $\Delta$ and containing $S$,
\begin{equation*}
	\Star{S}{\Delta}=\{A:\, A\in \bar{T},\, T\in \Delta,\, S\subseteq T\}.
\end{equation*}
\end{definition}

\noindent
We highlight when $S=\{i\}$ is a vertex, then $\Star{i}{\Delta}$ is the set of simplex $\bar{T}$ containing $i$, that is
\[
	\Star{i}{\Delta}=\{A:\, A \subseteq T,\, i\in T\in \Delta\}.
\]

\begin{definition}\label{def:link}
The link of an element $S$ in a simplicial complex $\Delta$ is made by the subsets $A$ of $T\in \Delta$, such that $T$ is disjoint by $S$ and can be completed by $S$, $S\cup T$, in $\Delta$:

\begin{equation*}
	\Link{S}{\Delta}=\{A:\, A\in \bar{T} \mbox{ with } T\in \Delta \mbox{ such that } S\cap T = \emptyset, S\cup T \in \Delta\}.
\end{equation*}
\end{definition}

\noindent
The case when $S$ is the singleton $\{i\}$ will be extremely relevant in our work: $\Link{i}{\Delta}$ is the set of simplex $T$ in $\Delta$ with $i\notin T$ such that $T\cup i\in \Delta$:
\[
	\Link{i}{\Delta}=\{T\in \Delta: i\notin T \mbox{ and } T\cup i \in \Delta\}.
\]

%
%
%
%
%

The $f$-vector of a simplicial complex is a integral vector $\ffD=(\ffi{-1}{\Delta}, \dots, \ffi{r-1}{\Delta})$ where
\begin{equation}\label{eq:def-f-vector}
	\ffi{i}{\Delta}\define\#\{S\in \Delta, \#S=i+1\}.
\end{equation}
If $\Delta\neq \emptyset$, then the empty set belongs to the complex and, thus, $\ffi{-1}{\Delta}=1$.
Just to give a concrete example, if $\Delta$ is a the full simplex on $n$ verticies, $\Delta=2^{[n]}$, then $\ffi{i}{\Delta}={n \choose i+1}$.

\subsection{Cooperative games on simplicial complexes}
A cooperative game on a simplicial complex $\Delta$ is the pair $(\Delta, v)$ where $v$ is a characteristic function $v:\Delta\rightarrow \mathbb{R}$ with $v(\emptyset)=0$. 
This notion has been defined in \cite{Martino-cooperative}, where the author generalizes the notion of cooperative game on matroids given by Bilbao, Driessen, Jim\'{e}nez Losada and Lebr\'{o}n \cite{Shapley-matroids-static}.

The verticies of $\Delta$ are the players of the cooperative game and a coalition $T$ is \emph{feasible} if $T\in \Delta$.
The set $\charFun$ of characteristic functions on $\Delta$ is a real vector space, similarly as in the traditional setting.
Indeed, given $(\Delta, v)$, one can re-scale the characteristic function with a scalar $c$ and obtain a new cooperative game $(\Delta, cv)$, where $(cv)(T)=c(v(T))$ for every subset $T\in \Delta$.

Each individual value $\phi_i(v)$ should assesses the worth of the participation of the player $i$ in to the game. 
Of course, we are looking for values such that $\phi_i(cv)=c\phi_i(v)$, because the worth of each player is just re-scaled.
For this reason, we consider a cone $\mathfrak{I}$ of cooperative game in $\charFun$.

\begin{definition}
An \emph{individual value} for a player $i$ in $[n]$ is a function $\phi_i:\mathfrak{I}\rightarrow \mathbb{R}$.
\end{definition}

\subsection{Carrier games}
There are two types of cooperative games having an central role in the theory of probalistic values \cite{Weber-robabilistic-values-for-games}. 
Despite in literature this terminology often refer to the first one, we are going to called the following \emph{carrier games}:
\[
	\mathcal{C} = \{v_T: \emptyset \neq T \subset [n] \}, \,\,\,\hat{\mathcal{C}} = \{\hat{v}_T: \emptyset \neq T \subset [n] \},
\]
where $v_T$ and $\hat{v}_T$ are so defined:
\[
	v_T (S)=\begin{cases}
				1 &  T\subseteq S\\
				0 &\text{otherwise.}
			\end{cases},\,\,\, \hat{v}_T (S)=\begin{cases}
				1 &  T\subsetneq S\\
				0 &\text{otherwise.}
			\end{cases}
\]
We generalize these games for any element $T$ of a simplicial complex. Indeed for every partially order set $(P, \leq_P)$ and every element $q$ in $P$ we consider the following function:
\[
	u_q^P(s)\define \begin{cases}
				1 &  q\leq_P s \\
				0 &\text{otherwise.}
			\end{cases},\,\,\, 				\hat{u}_q^P(s)\define \begin{cases}
				1 &  q<_P s \\
				0 &\text{otherwise.}
			\end{cases}
\]
Thus, we define
\[
	v_T (S)\define u_T^\Delta(S),\,\,\, \hat{v}_T (S)\define \hat{u}_T^\Delta (S).
\]
In the traditional setting, these functions reproduce the carrier games.

\begin{definition}
	Let $\Delta$ be a simplicial complex. The sets of carrier games are so defined:
	\[
		\mathcal{C} = \{v_T: \emptyset \neq T \in \Delta \}, \,\,\,\hat{\mathcal{C}} = \{\hat{v}_T: \emptyset \neq T\in \Delta \}, 
	\]
	where $v_T (S)\define u_T^\Delta(S)$ and $\hat{v}_T (S)\define \hat{u}_T^\Delta (S)$; moreover, $\hat{v}_{\emptyset}\define \hat{u}_{\emptyset}^\Delta$.
\end{definition}

In almost all of the main results in this manuscript we need to require that certain carrier games belong to $\mathfrak{I}$. In particular, we are going to write "\emph{Assume that the cone of cooperative games $\mathfrak{I}$ contains the carrier games $\mathcal{C}$ and $\hat{\mathcal{C}}$ for $\Star{i}{\Delta}$.}" and we mean that both $v_T$ and $\hat{v}_T$ belong to $\mathfrak{I}$ for all $T\in \Star{i}{\Delta}$.

\section{Coalitions Set and Linearity Axiom}\label{sec:linearity}

Before even starting with the axiomatization of the probabilistic values, we want to devote a few paragraphs to study the set of coalitions that a player $i$ may join in the game.

\noindent
Let $(\Delta, v)$ be a cooperative game on a simplicial complex. 
The coalition set $\Coalition{i}{\Delta}$ is the set of feasible coalitions that the player $i$ can join in the game. $\Coalition{i}{\Delta}$ is defined as a set, but in the next proposition 
we show it is a simplicial complex.

\begin{proposition}\label{prop:coalition-link}
	The coalition set $\Coalition{i}{\Delta}$ is a simplicial complex. Precisely:
	\begin{equation*}
		\Coalition{i}{\Delta}=\Link{i}{\Delta}.
	\end{equation*}		
\end{proposition}
\begin{proof}
	The statement follows naturally by definition of link.
\end{proof}

\noindent
In other words, the previous statement shows that the coalition set does only depends on the simplicial complex $\Delta$, i.e. it is independent by the characteristic function $v$.

Now, given a (positive) real scalar $c$, we may also consider the re-scaled game $(n, c\cdot v)$ given by the characteristic function $(c\cdot v)(T)=cv(T)$ for every $T\in \Delta$. 
Since in this new game, the worth of every coalition has been only re-scaled, it is natural that the new individual value $\phi_i(c\cdot v)$ would also be re-scaled by $c$, $\phi_i(c\cdot v)=c\phi_i(v)$. 

\begin{description}
	\item [From now on] we assume that the individual value function is defined in a cone $\mathfrak{I}$ of cooperative games in the vector space $\charFun$.
\end{description}


The first classical requirement is that the individual values are linear functions.

\begin{description}
	\item [Linearity Axiom] 
	The individual value function $\phi_i(v)$ of the player $i$ for the cooperative game $(\Delta, v)$ is linear on the cone $\mathfrak{I}$.
\end{description}


At this stage, it is worth to provide an interpretation for $\Star{i}{\Delta}$. 
This is in fact the completed set of coalitions for the player $i$, that is the set of coalition $T$ the player $i$ can join together with the formed coalitions $T\cup i$. 
	
\noindent
It is reasonable to that the individual value $\phi_i(v)$ only depends on values of the characteristic function $v$ on for the coalitions in $\Star{i}{\Delta}$. 
This is precisely the content of the next theorem.

\begin{TheoM}\label{thm:linearity-axiom-step}
	%
	Assume that the cone of cooperative games $\mathfrak{I}$ contains the carrier games $\mathcal{C}$ and $\hat{\mathcal{C}}$ for every $T\in\Delta$ but $T\cup i\notin \Delta$.
	%
	If the individual value $\phi_i$ fulfills the linearity axiom, 
	then, $\phi_i(v)$ only depends on the value of $v$ on the Star of $i$ in $\Delta$, $\Star{i}{\Delta}$, that is
	\begin{equation}\label{eq:linearity-axiom}
	\phi_i(v)=\sum_{T\in \Star{i}{\Delta}} a_T v(T).
\end{equation}
The set of coefficients $\{a_T\}$ is unique.	
\end{TheoM}


\begin{proof}
	Since the function $\phi_i$ is linear, then $\phi_i(v)=\sum_{T\in \Delta} a_T v(T)$ and the coefficients $a_T$ are unique.

	We need to prove that $a_T$ is zero if $T$ is not a set in $\Star{i}{\Delta}$. Specifically, we want to show that $a_T$ is zero if $T\in\Delta$ but $T\cup i\notin \Delta$.
	Let $\mathbbm{1}_T$ be the indicator function of the set $T$ in $\Delta$, that is:
	\[
		\mathbbm{1}_T(S)\define \begin{cases}
				1 &  T=S \in \Delta \\
				0 &\text{otherwise.}
			\end{cases}.
	\]
	
	Since the carrier games belong to $\mathfrak{I}$, then we write the indicator functions as $\mathbbm{1}_T=v_T-\hat{v}_T$. We note that $\phi_i(\mathbbm{1}_T)=a_Tv(T)$ and then $\phi_i(\mathbbm{1}_T)=\phi_i(v_T-\hat{v}_T)=a_Tv(T)=0$ if $T\cup i\notin \Delta$. 	
	This is because $\phi_i(\hat{v}_T)=\mathbbm{1}_{T\cup i}v(T\cup i)$ and because $\phi_i(v_T)=v_T(i)=0$.
\end{proof}

For now, it is convenient to rewrite $\phi_i(v)$ in \eqref{eq:linearity-axiom} as
\begin{equation}\label{eq:linearity-axiom-link}
	\phi_i(v)=\sum_{T\in \Link{i}{\Delta}} a_T v(T) +\sum_{T\in \Link{i}{\Delta}}a_{T\cup i} v(T\cup i).
\end{equation}

\section{Dummy Player Axiom}\label{sec:dummy}
The next natural condition we want to assume is that if a player does contribute the same specific amount $v(i)$ to every coalition, then her/his retribution should be precisely $v(i)$, that is the number of payoff he can reach by her/himself.
Such player is traditionally called \emph{dummy}, because he has no strategy, not matter what coalition he joins. This is the mathematical description of this phenomenon.

\noindent
A player $i$ is dummy for $v$ if for every set $T$ in the coalition set $\Coalition{i}{\Delta}$ one has
\[
	v(T\cup i)=v(T)+v(i).
\]

We are ready to state the so called dummy axiom:
\begin{description}
	\item [Dummy Player Axiom] 
If a player $i$ is dummy for the cooperative game $(\Delta, v)$, then its indi\-vidual value $\phi_i(v)$ is precisely $v({i})$.
\end{description}

\noindent
This axiom is a strong condition on the individual value and indeed provides the very traditional formula expressing $\phi_i(v)$ in terms of the marginal contributions $\left(v(T\cup i) - v(T)\right)$. 

\begin{TheoM}\label{thm:la+da}
	Assume that the cone of cooperative games $\mathfrak{I}$ contains the carrier games $\mathcal{C}$ for $\Link{i}{\Delta}$.
	If the individual value $\phi_i$ can be written as in equation \eqref{eq:linearity-axiom} and satisfies the dummy player axiom, then
	there exist real numbers 
	$\{p_T\}_{T\in \Link{i}{S}}$ 
	such that
	$$\sum_{T\in \Link{i}{\Delta}} p_T=1$$
	and for every $v$ in $\mathcal{I}$
	\begin{equation}\label{eq:012-phi-p_T}
		\phi_i(v)=\sum_{T\in \Link{i}{\Delta}}p_T \left(v(T\cup i) - v(T)\right).
	\end{equation}	
\end{TheoM}
\begin{proof}
	The structure of the proof follows similar as the one of Theorem 2 in \cite{Weber-robabilistic-values-for-games} but since there are several mathematical difference we write it here in details.
	
	We start by showing the second statement.
	First, we observe that for any nonempty $T\in \Link{i}{\Delta}$, the player $i$ is dummy for the carrier game $v_T$; moreover $\phi_i(v_T)=v_T(i)=0$, by definition of the game $v_T$.
	
	Now consider a facet $F_{(0)}$ in $\Link{i}{\Delta}$ and let us compute $\phi_i(v_{F_{(0)}})$. Since $F_{(0)}\cup i$ is also a facet for $\Delta$, one gets that $\phi_i(v_{F_{(0)}})=a_{F_{(0)}}+a_{F_{(0)}\cup i}$. In addition, $a_{F_{(0)}}+a_{F_{(0)}\cup i}=0$ because $i$ is a dummy player for the carrier game $v_{F_{(0)}}$. 
	
	\noindent
	Let us assume that $\phi_i(v_{F_{(j)}})=a_{F_{(j)}}+a_{F_{(j)}\cup i}=0$ for all face $F_{(j)}$ of co-rank $j$ less or equal to $k$ in $\Link{i}{\Delta}$ and let us prove that $a_{F_{(k+1)}}+a_{F_{(k+1)}\cup i}=0$ for every face $F_{(k+1)}$ of co-rank $k+1$.
	Indeed, by definition of $v_{F_{(k+1)}}$ we have that $\phi_i(v_{F_{(k+1)}})= \sum_{T \supseteq F_{(k+1)}, T\in \Star{i}{\Delta}} a_T$. This is equal to
	\[
		\sum_{j=0}^{k+1}\sum_{K} a_K
	\]	
	Where the second sum runs over the set $K\in \Star{i}{\Delta}$ of co-rank $j$, containing $F_{(k+1)}$, $K \supseteq F_{(k+1)}$. 

	\noindent	
	We then isolate the term $a_{F_{(k+1)}}+a_{F_{(k+1)}\cup i}$ 
	\[
		a_{F_{(k+1)}}+a_{F_{(k+1)}\cup i}+\sum_{j=0}^{\textbf{k}}\sum_{T'} a_{T'},
	\]
	and, by induction, the sum $\sum_{j=0}^{\textbf{k}}\sum_{T'} a_{T'}$ splits like $\sum_{j=0}^{\textbf{k}}\sum_{T} \left( a_{T}+a_{T\cup i}\right)$ and where $T$ runs over the set $T\in \Link{i}{\Delta}$ of co-rank $j$, containing $F_{(k+1)}$. Hence $\sum_{j=0}^{\textbf{k}}\sum_{T} \left( a_{T}+a_{T\cup i}\right)=0$ and consequently, $\phi_i(v_{F_{(k+1)}})=a_{F_{(k+1)}}+a_{F_{(k+1)}\cup i}=0$.
	
	
%
	
	\noindent
	To conclude the proof of the second statement, it remains to denote for every $T$ in $\Link{i}{\Delta}$ $$p_T\define a_{T\cup i}=-a_{T}.$$ 
	
	To prove the first part, instead, consider the carrier game $v_{i}$ in $\Link{i}{\Delta}$ and observe that $i$ is a dummy player in this game, so 
	\[
		\phi_{i}(v_i)=v_i(i)=1.
	\]
	Using \eqref{eq:012-phi-p_T}, that we have just proved, one has
 	\[
 		\phi_i(v)=\sum_{T\in \Link{i}{\Delta}}p_T \left(v(T\cup i) - v(T)\right)=\sum_{T\in \Link{i}{\Delta}}p_T=1.
 	\]
\end{proof}

\begin{remark}
There is no restriction on the real number $p_T$, so a-priory, $p_T$ may also be negative.
\end{remark}

\begin{remark}
It is worth to clarify our notation in the case $T=\emptyset$: $p_{\emptyset}=a_{i}=-a_{\emptyset}$.
\end{remark}


\section{Monotonicity Axiom}

The third classical condition required for an individual axiom is the monotonicity. This deals with a class of very special functions on the simplicial complex, the monotone functions: $v:\Delta \rightarrow \mathbb{R}$ such that for every $S\subseteq T\in \Delta$, then $v(S)\leq v(T)$. 
%


\begin{description}
	\item [Monotonicity Axiom] 
	If $v$ is a monotone game, then $\phi_i(v)\geq 0$ for each player $i$.
\end{description}

In other words, if $v$ is a monotone function, the player $i$ knows that his marginal contribution to the game in every coalition will be always positive, that is $v(T\cup i)\geq v(T)$. Thus, the player is going to require a positive retribution. 

In view of what we have just proven in Theorem \ref{thm:la+da}, or by assuming that the individual value $\phi_i$ can written as in equation \eqref{eq:012-phi-p_T}, we immediately show that 
the real constants $\{p_T\}_{T\in \Link{i}{\Delta}}$ are actually always non-negative and, hence, they are a probability distribution on the elements of $\Link{i}{\Delta}$.

\begin{TheoM}\label{thm:la+da+ma}
	Assume that the cone of cooperative games $\mathfrak{I}$ contains the carrier games $\mathcal{C}$ and $\hat{\mathcal{C}}$ for $\Star{i}{\Delta}$.
	We assume that the individual value $\phi_i$ can written as in equation \eqref{eq:012-phi-p_T}, that $\sum_{T\in \Link{i}{\Delta}} p_T=1$, and that it
	fulfills 
	the monotonicity axiom.
	Then, the constants $\{p_T\}_{T\in \Link{i}{\Delta}}$ are a probability distribution.
\end{TheoM}
\begin{proof}
	The the game $\hat{v}_T$ is monotone for all $T\in \Star{i}{\Delta}$. Then $\phi_i(\hat{v}_T)=p_T\geq 0$.
\end{proof}


We close this section by giving the main definition of this work.

\begin{definition}
Given a cone $\mathfrak{I}$ in the vector space $\charFun$ of cooperative games on a simplicial complex $\Delta$.
A \emph{probabilistic value} is a linear function $\phi_i$ written in the form $\eqref{eq:012-phi-p_T}$ such that $\{p_T\}_{T\in \Link{i}{\Delta}}$ is a probability distribution on the elements of link of the vertex $i$.
\end{definition}

Of course, Theorem \ref{thm:linearity-axiom-step}, Theorem \ref{thm:la+da} and Theorem \ref{thm:la+da+ma} provides all together a characterization for probabilistic values on the simplicial complex $\Delta$. One only needs to take care of all the requested carrier games in $\mathfrak{I}$.

\section{Group value and Symmetry Axiom}\label{Sec:symmetry-axiom}

To move forward in generalizing the classical theory for probabilistic values, we now introduce the notion of group value, the collection of the individual value functions for all players in the game. In other words:

\begin{definition}
A group value $\operatorname{\phi}$ for the cooperative game $(\Delta, v)$ is the vector $(\phi_1(v), \dots, \phi_n(v))$.
\end{definition}

To introduce the next requirement we need a further step. 
Consider the subgroup $\operatorname{Symm}(\Delta)$ of the symmetric group $S_n$ made by the permutations such that if $T\in \Delta$, then  $\pi T\in \Delta$.
%
Finally, let us denote by $\pi \cdot v$ the permuted characteristic function; this is $(\pi \cdot v): \Delta \rightarrow \mathbb{R}$ such that
\[
	\pi \cdot v(T)=	v(\pi T).
\]

We denote by $\pi_{L,T}^i$ the permutation that switches two set $L$ and $T$ of the same cardinality in $\Link{i}{\Delta}$ and fixes $i$. We define $\pi_{i,j}$ to be the simple transposition $(i,j)$.

\begin{description}
	\item [Symmetry Axiom] 
	Assume that for every permutation $\pi$ in $\operatorname{Symm}(\Delta)$ we have $\pi \cdot v$ belongs to $\mathfrak{I}$.
	Let $\phi$ be a group value for the cooperative game $(\Delta, v)$.
	Then, 
	$\phi_i(v)=\phi_{\pi i}(\pi \cdot v),$
	for every $\pi$ in $\operatorname{Symm}(\Delta)$.
\end{description}


%

Some of the symmetries in $S_n$ have a specific role in next proof. Let us recall them here.

\begin{definition}\label{def:subgroup}
	Let $\subgroupD$ be the subgroup of $S_n$ generated by the permutation $\pi_{L,T}^i$ for all $i$ and for all $L,T\in\Link{i}{\Delta}$ with $|L|=|T|$ and generated by every transposition $\pi_{i,j}$ provided that $\Link{i}{\Delta}\cap \Link{j}{\Delta}\neq \emptyset$.
\end{definition}

\begin{theorem}\label{thm:s.a.}
Let $\Delta$ be a simplicial complex and let $\mathcal{I}$ be a cone of games defined on $\Delta$, containing the carrier games $\mathcal{C}$ and $\hat{\mathcal{C}}$ for $\Link{i}{\Delta}$ for every vertex $i$.
Assume that if $v\in\mathfrak{I}$, then the permuted game $\pi \cdot v$ is also in $\mathfrak{I}$ for every $\pi$ in $\operatorname{Symm}(\Delta)$.

Let $\phi$ be a group value such that for each $i\in [n]$ and for each $v\in\mathfrak{I}$, we can write:
\[
		\phi_i(v)=\sum_{T\in \Link{i}{\Delta}}p_T^i \left(v(T\cup i) - v(T)\right).
\]

\noindent
If $\subgroupD \subseteq \operatorname{Symm}(\Delta)$ and if $\phi$ satisfies the symmetric axiom, then there are real constants $p_t$ with $\mathbf{t\neq 0}, \dots, \operatorname{rank}\Delta - 1$ such that for all players $i$ and all \textbf{non-trivial} $T\in \Link{i}{\Delta}$ one has $p_T^i=p_{|T|}$.
\end{theorem}

%
%

\begin{proof}
	We first note that the rank of $\Link{i}{\Delta}$ is at most (but not necessarily the same) $\operatorname{rank}\Delta - 1$. 
	  
	
	Fix $i$ and pick two non-trivial sets (but not facets) $L$ and $T$ of the same cardinality in $\Link{i}{\Delta}$. 	From the hypothesis $\pi_{L,T}^i\in \operatorname{Symm}(\Delta)$ and 
	then, because of the symmetry axiom,
	$p_T^i=\phi_i(\hat{v}_{T})=\phi_i(\pi_{L,T}^i \cdot \hat{v}_{T})=\phi_i(\hat{v}_{L})=p_L^i$.

	Now, for every $i$ and consider two facets $F$ and $F'$ of the same cardinality in $\Link{i}{\Delta}$.
	%
	%
	One has $\pi_{F,F'}^i\in \operatorname{Symm}(\Delta)$ and because of the S.A we get, $p_F^i=\phi_i(v_{F\cup i})=\phi_{\pi_{F,F'}^i(i)}(\pi_{F,F'}^i \cdot v_{F\cup i})=\phi_i(v_{F'\cup i})=p_{F'}^i$.

	Finally, consider two distinct player $i$ and $j$ and pick $T$ in $\Link{i}{\Delta}\cap \Link{i}{\Delta}$. Then, the transposition $\pi_{i,j}=(i,j)$ belongs to $\operatorname{Symm}(\Delta)$ and we obtain $p_T^i=\phi_i(\hat{v}_{T})=\phi_{\pi_{i,j}=(i)}(\pi_{i,j} \cdot \hat{v}_{T})=\phi_{j}(\hat{v}_{T})=p_T^j$.
\end{proof}

\begin{remark}\label{rmk:weaker-the-hypothesis}
    We could have request less carrier games in the cone $\mathfrak{I}$: specifically we only needs that 
    the carrier games in $\hat{\mathcal{C}}$ belongs to $\mathfrak{I}$ for every non-facets in the intersection of two links and for every non-facet $T$ in the link with at least another element of the same cardinality (in the same link).
    We also need that the carrier games in $\mathcal{C}$ for every facet in the link with at least another facet of the same cardinality (in the same link).
    
    The requests of the theorem are stronger because we aim to simplify the hypothesis. 
\end{remark}

\begin{remark}
We observe that a priori such constants $p_t$ could be also negative, but if we work with monotone games then we get as well that the $p_t$'s are non-negative.
\end{remark}

As harmless as it could seem, passing from the symmetric group $S_n$ to the subgroup of symmetries $\operatorname{Symm}(\Delta)$ leads us already to requiring that a specific set of symmetries $\subgroupD$ belong to $\operatorname{Symm}(\Delta)$.
%
A more careful look at the last paragraph of the previous proof highlights that the consequences are drastically stronger than in the traditional case. 
In the next section we are going to deal with this issue, together with the problem of assigning a common value $p_0$ for the probability $p_{\emptyset}^i$ of every player.

\subsection{Simplicial complex having pure links}\label{sec:restriction-for-equidim-links}

Let us analyze the implications of Theorem \ref{thm:s.a.}. Assume that all the hypothesis of the previous theorem are satisfied and that for two distinct players $i$ and $j$ the rank of $\Link{i}{\Delta}$ is strictly smaller than the rank of $\Link{j}{\Delta}$, say $r_i < r_j$.
Then, Theorem \ref{thm:s.a.} says that $p^{j}_T$ has to be zero for any subset of cardinality larger than $r_i$. Indeed there exist a common value $p_{|T|}$ for all probabilities $p^{j}_T$ of $j$ joining a cardinality $|T|$ subset $T$. In facts, there are no subsets of cardinality $|T|>r_i$ in $\Link{i}{\Delta}$, because of the rank restriction. 
Thus, $p_{|T|}=0$ seen from the point of view of $i$.

In other words, the symmetry axiom implies that the strategy constrains are shared among all players: in particular the probability of joining a certain coalition of size $|T|$ has to be shared to; If for a player there is no chance of joining a certain coalition of cardinality $|T|$, so it must be also for the other players.

Now, the theory we develop allows $p^{j}_T=0$ even in the case that $T$ is a feasible coalition. On the other hand, in this specific case, if $p^{j}_T=0$, for all subsets of cardinality strictly larger than $k$, then one can simply reconsider the same cooperative game defined on the $k$-skeleton of simplicial complex $\Delta$, that is $\Delta^{\langle k\rangle}\define \{S\in \Delta: |S|\leq k \}$.

Therefore, we are not losing any generality if we assume the symmetry axiom and we work with simplicial complexes \emph{having pure links}:

\begin{definition}\label{def:pure-links-simplicial-complex}
    A simplicial complex $\Delta$ of rank $r$ has \emph{pure links} if the link of every vertex $i$, $\Link{i}{\Delta}$, is pure of rank $r-1$.
\end{definition}

\subsection{Solving a linear system}\label{sec:bit-system-f-vector}

We need to carry this analysis further to answer: is there a common value $p_0$ for the probabilities $p^i_{\emptyset}$?
As done in the end of the proof of Theorem \ref{thm:la+da}, consider the carrier game $v_{i}$ in $\Star{i}{\Delta}$ and observe that $i$ is a dummy player in this game, so 
\[
	\phi_{i}(v_i)=v_i(i)=1.
\]
Using \eqref{eq:012-phi-p_T}, that is also an hypothesis in Theorem \ref{thm:s.a.}, one has
\[
 	\phi_i(v)=\sum_{T\in \Link{i}{\Delta}}p_T \left(v(T\cup i) - v(T)\right)=\sum_{T\in \Link{i}{\Delta}}p_T=1.
\]
So for every $i$,
\[
	p^i_{\emptyset}=1-\sum_{\emptyset \neq T\in \Link{i}{\Delta}}p_T^i.
\] 
For what we have just proven in Theorem \ref{thm:s.a.}, the previous equation becomes
\[
	p^i_{\emptyset}=1-\sum_{\emptyset \neq T\in \Link{i}{\Delta}}p_{|T|}.
\] 
Keeping track of the cardinality of $T$ and using the $f$-vector notation, see equation \eqref{eq:def-f-vector}, we obtain
\[
	p^i_{\emptyset}=1-\sum_{j=0}^{\rank{\Link{i}{\Delta}}}\ffi{j-1}{\Link{i}{\Delta}} p_{j}
\] 
Recalling that $\ffi{-1}{\Link{i}{\Delta}}=1$, see Section \ref{sec:preliminaries}, and imposing that all $p^i_{\emptyset}$ are equal to $p_0$, we finally get  
\begin{equation}
	\sum_{k=0}^{\rank{\Link{i}{\Delta}}}\ffi{k-1}{\Link{i}{\Delta}} p_{k} =1.
\end{equation}
Finally we assume that we work with a simplicial complex of rank $r$ with pure links, see Definition \ref{def:pure-links-simplicial-complex}, and we finally get
\begin{equation}\label{eq:symmetric-for-emptyset}
	\sum_{k=0}^{r-1}\ffi{k-1}{\Link{i}{\Delta}} p_{k} =1.
\end{equation}
Because of the labeling shift, it is more elegant to write this in the scalar product form $$\ff{\Link{i}{\Delta}}\cdot \textbf{p}=1.$$

\begin{TheoM}\label{thm:s-a-unique-p-0}
	Let $\Delta$ be a simplicial complex of rank $r$ \textbf{with pure links} and let $\mathfrak{I}$ be a cone of cooperative games defined on $\Delta$, containing the carrier games $\mathcal{C}$ and $\hat{\mathcal{C}}$ for the star $\Star{i}{\Delta}$ for every vertex $i$.
    Assume that if $v\in\mathfrak{I}$, then the permuted game $\pi \cdot v$ is also in $\mathfrak{I}$ for every $\pi$ in $\operatorname{Symm}(\Delta)$.
    Let $\phi$ be a group value such that for each $i\in [n]$ and for each $v\in\mathfrak{I}$, we can write:
    \[
		\phi_i(v)=\sum_{T\in \Link{i}{\Delta}}p_T^i \left(v(T\cup i) - v(T)\right).
    \]

    If $\subgroupD \subseteq \operatorname{Symm}(\Delta)$ and if $\phi$ satisfies the symmetric axiom, then there are real constants $p_t$ with $\mathbf{t \neq 0}, \dots, \operatorname{rank}\Delta - 1$ such that for every players $i$ and every \emph{non-trivial} $T\in \Link{i}{\Delta}$ one has $p_T^i=p_{|T|}$.

	Moreover, there exists a common value $p_0$ for the probabilities $p_{\emptyset}^i$ if $\textbf{p}=(p_0, p_1, \dots, p_{r-1})$ is the solution of the following system of $n$ linear equations:
	{\em
	\begin{equation}\label{eq:system-for-unique-p}
	    \left\{ \begin{array}{c}
	         \ff{\Link{1}{\Delta}}\cdot \textbf{p}=1,\\
	         \ff{\Link{2}{\Delta}}\cdot \textbf{p}=1,\\
	          \vdots \\
	         \ff{\Link{n}{\Delta}}\cdot \textbf{p}=1,
	    \end{array} \right.
	\end{equation}}
	where each equation has the form in \eqref{eq:symmetric-for-emptyset}. 
\end{TheoM}

\subsection{The Shapley complexes}\label{sec:restriction-same-f-vector}

The first thing to observe is that many of the equations in \eqref{eq:system-for-unique-p} are actually the same.
This is a side effect of the hypothesis $\subgroupD \subseteq \operatorname{Symm}(\Delta)$.

\begin{proposition}
	Let the transposition $\pi_{i,j}$ belong to $\operatorname{Symm}(\Delta)$ if $\Link{i}{\Delta}\cap \Link{j}{\Delta}\neq \emptyset$. Then $\Link{i}{\Delta}$ is isomorphic to $\Link{j}{\Delta}$.
	Moreover, {\em $\ff{\Link{i}{\Delta}}=\ff{\Link{j}{\Delta}}$}.
\end{proposition}
\begin{proof}
	A face $T$ belongs to $\Link{i}{\Delta}$ if and only if $T\cup i\in \Delta$ and $i\notin T$. 
	Note $(i,j)T=T$ and $(i,j)T\cup i=T\cup j$. 
	Moreover, if $j\in T$ and $S=(T\cup i)\setminus j$, then $(i,j)T=S$ and, of course, $(i,j)S=T$.
	
	Thus the isomorphism in the statement is $\phi_{(i,j)}$ induced by $(i,j)$:
	\[
	\begin{array}{cccc}
		\phi_{(i,j)}:&\Link{i}{\Delta} & \rightarrow & \Link{j}{\Delta}\\
		 &T	&\mapsto & \begin{cases}
		                    T & \mbox{ if } j\notin T\\
		                    (T\cup i)\setminus j & \mbox{ if } j\in T
		               \end{cases}.
	\end{array}
    \]
	
\end{proof}


If the simplicial complex has enough faces then it may very well happen that several links intersect and we obtain that all $\ff{\Link{i}{\Delta}}$ are actually the same. This is anyhow the case we are interested.
Indeed we would like to understand the class of simplicial complex such that the system of equation in \eqref{eq:system-for-unique-p} has at least a solution.
%

\begin{definition}\label{def:Shapley-simplicial-complex}
	A simplicial complex $\Delta$ of rank $r$ on $n$ verticies is  \textbf{Shapley} of vector $\textbf{s}=(s_0,\dots, s_{r-1})$, if $\ff{\Link{i}{\Delta}}=\textbf{s}$ for every vertex $i$.
\end{definition}


\begin{example} 
The full simplex on $n$ verticies $\Delta=2^{[n]}$ is Shapley of vector $\textbf{s}$, with $s_i={n \choose i}$ for $i=0, \dots, n-1$.
\end{example}

\begin{example} 
The boundary of a simplex on $n$ verticies $\Delta=\partial(2^{[n]})$ is Shapley of vector $\textbf{s}$, with $s_i={n \choose i}$, for $i=0, \dots, n-2$. Similarly, $\Delta^{\langle k\rangle}$ is a Shapley complex of vector $\left({n \choose 0}, {n \choose 1}, \dots, {n \choose k-1}\right)$.
\end{example}

\begin{example} 
Let $G$ be a graph on $n$ verticies, then the $f$-vector of each links reduce to only one non-trivial information: the number of edges to the selected vertex. Thus, $G$ is Shapley of vector $(1, s)$ if and only if $G$ is $s$-regular.
\end{example}

Ee can ensure the existence of at least a solution of the system \eqref{eq:system-for-unique-p} for Shapley simplicial complexes.

\begin{theorem}\label{thm:at-least-a-solution-shapley-complex}
		Let $\Delta$ be a \textbf{Shapley} simplicial complex of vector $\textbf{s}=(s_0, \dots, s_{r-1})$ and and let $\mathfrak{I}$ be a cone of games defined on $\Delta$, containing the carrier games $\mathcal{C}$ and $\hat{\mathcal{C}}$ for the star $\Star{i}{\Delta}$ for every player $i$.
    Assume that if $v\in\mathfrak{I}$, then the permuted game $\pi \cdot v$ is also in $\mathfrak{I}$ for every $\pi$ in $\operatorname{Symm}(\Delta)$.

    Let $\phi$ be a group value such that for each $i\in [n]$ and for each $v\in\mathfrak{I}$, we can write:
    \[
		\phi_i(v)=\sum_{T\in \Link{i}{\Delta}}p_T^i \left(v(T\cup i) - v(T)\right).
    \]

    \noindent
    If $\subgroupD \subseteq \operatorname{Symm}(\Delta)$ and if $\phi$ satisfies symmetric axiom, then there are real constants $p_t$ with $\mathbf{t = 0}, \dots, r - 1$ such that $p_T^i=p_{|T|}$ for all players $i$ and all $T\in \Link{i}{\Delta}$.
\end{theorem}
\begin{proof}
	We only need to provide a solution for the system \eqref{eq:system-for-unique-p}, that for Shapley complexes is made of only one equation. A solution is:
	\[
		p_{k}=\frac{1}{r}\frac{1}{s_k}.
	\]
\end{proof}

\section{The efficiency axiom and the Shapley value}\label{sec:Shapley}

The Shapley values are probabilistic values for the cooperative game $(2^{[n]}, v)$ arising from the following point of view:
\begin{itemize}
	\item The player $i$ joins a coalitions of different sizes with the same probability;
	\item All coalitions of the same size are equally likely. 
\end{itemize}
Thus, the player has $n$ possibilities to choose the size of a coalition (the joint coalitions may have cardinality $k= 0\leq k\leq n-1$) and, further, there are ${n-1 \choose k}$ choices among all sets (coalitions) of cardinality $k$ among the set of other players $[n]\setminus i$.
Therefore, one defines the Shapley values for the player $i$ as 
\begin{equation}\label{eq:shapley-classical-expression}
    \operatorname{Shapley}_i(v)=\sum_{T\subseteq [n]\setminus i} \frac{1}{n}\frac{|T|!(n-|T|-1)!}{(n-1)!} (v(T\cup i)-v(T)).
\end{equation}

To state the classical Shapley Theorem, we need to introduce the last requirement: efficiency.
The individual value $\phi_i$ associated to a cooperative game $(2^{[n]}, v)$ measures the contribution of the player $i$ in the game. We have collected such values all together in the group value $\phi=(\phi_1, \phi_2, \dots, \phi_n)$. The assessment is optimistic (w. r. to $v$) if the sum of the payoff vector $\sum_i \phi_i(v)$ is greater than the $v([n])$, the worth of the grand coalition. If the contrary happens, then $\phi$ is pessimistic (w. r. to $v$).
We are interesting in group values that are nor optimistic or pessimistic.

\begin{description}
	\item [Efficiency Axiom] 
	Let $\phi$ be a group value for the cooperative game $(2^{[n]}, v)$ in the cone $\mathfrak{I}$.
	For every cooperative game $(2^{[n]}, v)$ in $\mathfrak{I}$, one has $\sum_{i=1}^n \phi_i(v)=v([n])$. 
\end{description}

Despite this axiom might seem artificial, this is the missing piece to characterize the Shapley values in the traditional setting.

\begin{ShapleyThm*}
	Let $\Delta=2^{[n]}$ and let $\mathfrak{I}$ be a cone of cooperative games containing the carrier games $\mathcal{C}$ and $\hat{\mathcal{C}}$ for $N\setminus i$.
	Assume that if $v\in\mathfrak{I}$, then the permuted game $\pi \cdot v$ is also in $\mathfrak{I}$ for every permutation $\pi$ of $[n]$.

	Let $\phi$ be a group value. If each $\phi_i$ satisfies the linearity axiom, the dummy player axiom 
	and if the symmetric axiom and the efficiency axiom hold for the group value $\phi$, then for every cooperative game in $\mathfrak{I}$ and every $i$ in $[n]$, 
	\[
		\phi_i(v)=\operatorname{Shapley}_i(v).
	\]
\end{ShapleyThm*}

If $\Delta$ is not the full simplex, then the the grand coalition $[n]$ is a forbidden coalition and, thus, one needs to study what could take the place of the the total number of payoff $\vtot=v([n])$.
This study is the center of attention of the manuscript \cite{Martino-Efficiency} and for what concerns us, we only need to recall Theorem 2.1 of \cite{Martino-Efficiency}.

\begin{theorem}[Theorem 2.1 of \cite{Martino-Efficiency}]
\label{thm:generic-efficiency}
Let $\Delta$ be a simplicial complex and let $\mathfrak{I}$ be a cone of cooperative games defined on $\Delta$ containing the carrier games $\mathcal{C}$ and $\hat{\mathcal{C}}$ for $\Delta$.

Let $\phi$ be a group value on $\mathfrak{I}$ such that for each $i\in [n]$ and assume that for each $v\in\mathfrak{I}$, we can write:
\[
		\phi_i(v)=\sum_{T\in \Link{i}{\Delta}}p_T^i \left(v(T\cup i) - v(T)\right).
\]

\noindent
The group value $\phi$ satisfies 
$$\sum_{i\in [n]}\phi_i(v)=\sum_{T\in \Delta} a_T v(T)$$
if and only if for all non-facet $T$ in $\Delta$
\begin{equation*}
\sum_{i\in T}p^i_{T\setminus i}-\sum_{j, T\in \Link{j}{\Delta}}p_T^j=a_T,
\end{equation*}
and for every facet $F$ of $\Delta$
\begin{equation*}
	\sum_{i\in F}p^i_{F\setminus i}=a_F.
\end{equation*}
\end{theorem}


%

\subsection{The Shapley values for simplicial complexes}
Following the same probabilistic argument we can easily generalize the Shapley value to every simplicial complex $\Delta$ by observing that, under the same prospective:
\begin{itemize}
    \item coalitions may have cardinality $k$, $0\leq k\leq r_i\define \operatorname{rank}\Link{i}{\Delta}$;
    \item the $f$-vector component $\ffi{k-1}{\Link{i}{\Delta}}$ provides the number of faces of a specific cardinality $k$,  see \eqref{eq:def-f-vector}.
\end{itemize}
%
Hence, we define the Shapley value for the cooperative game $v$ on a simplicial complex $\Delta$ as
\begin{equation}\label{eq:Shapley-def-simplicial-complex}
	\operatorname{Shapley}^{\Delta}_i(v)=\frac{1}{r_i+1} \sum_{T\in \Link{i}{\Delta}}  \frac{1}{\ffi{|T|-1}{\Link{i}{\Delta}}} (v(T\cup i)-v(T)).
\end{equation}

\noindent
We note that if $\Delta$ is a simplex, than $f_{|T|-1}$ is precisely $\frac{|T|!(n-|T|-1)!}{(n-1)!}$, and $r_i+1=n$; so $\operatorname{Shapley}^{\Delta}_i(v)=\operatorname{Shapley}_i(v)$ in the classical sense.


Let us recall a few relevant facts we will be useful for the main theorem in this section.

\vspace{0.1cm}
\textbf{i)} If $\mathfrak{I}$ is a cone of cooperative game over a simplicial complex $\Delta$ of rank $r$ with pure links fulfilling the symmetry axiom then without loss of generality we may assume that each link is pure and has the same rank, precisely $r_i=r-1$ for all $i\in [n]$, see Section \ref{sec:restriction-for-equidim-links}.

\vspace{0.1cm}
\textbf{ii)} Let the group value $\phi$ be made by Shapley values in \eqref{eq:Shapley-def-simplicial-complex}, and let $\Delta$ be a Shapley simplicial complex  of vector $\textbf{s}=(s_0, \dots, s_{r-1})$, that is $\ffi{k-1}{\Link{i}{\Delta}}=s_k$ for all $i\in [n]$, see Definition \ref{def:Shapley-simplicial-complex}.
Then, 
\[
		\operatorname{Shapley}^{\Delta}_i(v)=\frac{1}{r} \sum_{T\in \Link{i}{\Delta}}  \frac{1}{s_{|T|}} (v(T\cup i)-v(T)).
\]

\vspace{0.1cm}
\textbf{iii)} Moreover, we define a player $j$ to be an extension for the coalition $T$ if $j$ can join $T$ for the cooperative game:
\begin{equation*}
	\operatorname{Ext}(T)=\{j\in [n]\setminus T: T\cup j \in \Delta \};
\end{equation*}
in addition, $\operatorname{ext}(T)=|\operatorname{Ext}(T)|$.

\begin{lemma}
	Let $\Delta$ be a simplicial complex of rank $r$. Then, $\operatorname{Ext}(T)=\{j\in [n]: T\in \Link{j}{\Delta}\}$ and $\operatorname{ext}(T)=\ffi{0}{\Link{T}{\Delta}}$.
\end{lemma}
\begin{proof}
	The statement follows directly from the definitions of extension of $T$, link of $j$ and $f$-vector.
\end{proof}

%

We are therefore ready to provide an axiomatic description of the Shapley values for simplicial complexes.

\begin{TheoM}[Generalized Shapley's Theorem]
	Let $\Delta$ be a simplicial complex of rank $r$ with pure links and let $\mathfrak{I}$ be a cone of cooperative games on $\Delta$ containing the carrier games $\mathcal{C}$ and $\hat{\mathcal{C}}$ for the star $\Star{i}{\Delta}$ for every vertex $i$.
Assume that if $v\in\mathfrak{I}$, then the permuted game $\pi \cdot v$ is also in $\mathfrak{I}$ for every permutation $\pi$ in $\operatorname{Symm}(\Delta)$.

Let $\phi$ be a group value and assume that each $\phi_i$ satisfies the linearity axiom, the dummy player axiom. 
Assume also that 
the group value fulfills the symmetric axiom with the further constrain that every player enters the games with probability $p_0$, that is there exists a solution $\textbf{p}=(p_0, p_1, \dots, p_{r-1})$ for the linear system $n$ equations:
{\em
	\begin{equation*}
        \left\{ \begin{array}{c}
	         \ff{\Link{1}{\Delta}}\cdot \textbf{p}=1,\\
	         \ff{\Link{2}{\Delta}}\cdot \textbf{p}=1,\\
	          \vdots \\
	         \ff{\Link{n}{\Delta}}\cdot \textbf{p}=1,
	    \end{array} \right.
	\end{equation*}}
	where each equation has the form in  \eqref{eq:symmetric-for-emptyset}.

Then, every individual value is the Shapley value in equation \eqref{eq:Shapley-def-simplicial-complex} if and only if $\Delta$ is $\textbf{s}$-Shapley with $\textbf{s}=(s_0,\dots, s_{r-1})$ and the group value satisfies the following efficiency scenario:
\begin{equation}\label{eq:shapley-efficiency-condition}
		\sum_{i\in [n]}\phi_i(v)=\sum_{F\in \FacetsD} %
	\left( \frac{1}{s_{r-1}} \right)	v(F) 
	+\frac{1}{r}\sum_{T\in \Delta, |T|<r} 
	\left(\frac{|T|}{s_{|T|-1}} -\frac{\operatorname{ext}(T)}{s_{|T|}}\right) 
	v(T).
\end{equation}
\end{TheoM}
\begin{proof}
	Using Theorem \ref{thm:la+da}, we can write each $\phi_i$ as 
	\[
		\phi_i(v)=\sum_{T\in \Link{i}{\Delta}}p_T^i \left(v(T\cup i) - v(T)\right)
	\]
	and by using Theorem \ref{thm:s-a-unique-p-0} we know that the real numbers $p_T^i$ depend only by the cardinality of the set \textbf{$|T|=t$}, so
	\[
		\phi_i(v)=\sum_{T\in \Link{i}{\Delta}}p_{|T|} \left(v(T\cup i) - v(T)\right).
	\]
	If $\phi_i=\operatorname{Shapley}_i^{\Delta}$ for every $i$, then $p_t^i=\nicefrac{1}{r\ffi{t-1}{\Link{i}{\Delta}}}$; hence $\nicefrac{1}{r\ffi{t-1}{\Link{i}{\Delta}}}=\nicefrac{1}{r\ffi{t-1}{\Link{j}{\Delta}}}$ for every distinct pair of player and, then $\Delta$ is $\mathbf{s}$-Shapley with $s_k=\ffi{k-1}{\Link{i}{\Delta}}$.
		
	By using Theorem \ref{thm:generic-efficiency}, it is easy to see that if the group value $\phi$ is made of individual values of the form \eqref{eq:Shapley-def-simplicial-complex}, then $\sum_{i\in [n]}\operatorname{Shapley}_i^{\Delta}(v)$ is the number the one expressed in \eqref{eq:shapley-efficiency-condition}.
	
	On the contrary direction, assume $\Delta$ is $\textbf{s}$-Shapley and that \eqref{eq:shapley-efficiency-condition} holds.
	As we mention right before the statement of the theorem, \eqref{eq:Shapley-def-simplicial-complex} reduces to
	\[
			\operatorname{Shapley}^{\Delta}_i(v)= \sum_{T\in \Link{i}{\Delta}} \frac{1}{r} \frac{1}{s_{|T|}} (v(T\cup i)-v(T)).
	\]
	Consider the cooperative game $\mathbbm{1}_T=v_T-\hat{v}_T$, so $\phi_i(\mathbbm{1}_T)=\phi_i(v_T-\hat{v}_T)$.
	Hence, if $F$ is a facet for $\Delta$, then
	$$\sum_{i\in [n]}\phi_i^{\Delta}(\mathbbm{1}_T)=	\sum_{i\in F}p_{|F|-1},$$
	Now, by hypothesis $\sum_{i\in F}p_{|F|-1}=\frac{1}{s_{r}}$ and so $p_{r-1}=p_{|F|-1}=\frac{1}{r}\frac{1}{s_{r}}$.
	
	If $T$ is not a facet then
	$$\sum_{i\in [n]}\phi_i^{\Delta}(\mathbbm{1}_T)=\sum_{i\in [n]}\phi_i^{\Delta}(v_T)-\sum_{i\in [n]}\phi_i^{\Delta}(\hat{v}_T)=	\sum_{i\in T}p_{|T|-1}-\sum_{j\in \operatorname{Ext}(T)}p_{|T|}$$
	and, therefore, by hypothesis
	$$|T|p_{|T|-1}-\operatorname{ext}(T)p_{|T|}=\frac{1}{r}\left(\frac{|T|}{s_{|T|-1}} -\frac{\operatorname{ext}(T)}{s_{|T|}}\right).$$
	Using that $p_{r-1}=\frac{1}{r}\frac{1}{s_{r}}$, we recursively show that $p_{t}=\frac{1}{r}\frac{1}{s_{t}}$.
\end{proof}

%
%
%
%
%

\section{Decomposing the Shapley value}\label{sec:decomposing}
In this section, we are going to study when the Shapley value 
$\operatorname{Shapley}^{\Delta}_i$
can be written as a weighted sum of classical Shapley values. Bilbao, Driessen, Jim\'{e}nez Losada and Lebr\'{o}n \cite{Shapley-matroids-static} characterized when individual values can be written as weighted sum of classical Shapley values defined on the maximal facets of a matroids \cite{Shapley-matroids-static}; the author has generalized this result to simplicial complex in Theorem 5.1 of \cite{Martino-cooperative}.
%
It is interesting to connect these two axiomatizations.

To do this, let $F$ be a facet of $\Delta$ and we denote by $\operatorname{Shapley}^{F}_i(v)$ the classical Shapley value for the cooperative game $v_{|F}$, that is a cooperative game on $|F|$ players where $v_{|F}(S)=v(S)$ for any subset $S$ of $F$.
Consider the following individual value
\[
	\operatorname{S}^{\Delta}_i(v)=\sum_{F\in\Facet{i}{\Delta}} c_F \operatorname{Shapley}^{F}_i(v_{|F})
\]
where $\{c_F\}$ is a subset of real numbers.

We start by substituting in the left hand side the expression \eqref{eq:shapley-classical-expression}: 
\begin{eqnarray*}
\operatorname{S}^{\Delta}_i(v)&=&\sum_{F\in\Facet{i}{\Delta}} c_F \operatorname{Shapley}^{F}_i(v_{|F})\\
&=&\sum_{F\in\Facet{i}{\Delta}} c_F  \sum_{T\subseteq F\setminus i} \frac{1}{|F|}\frac{|T|!(|F|-|T|-1)!}{(|F|-1)!} (v(T\cup i)-v(T))\\
&=&\sum_{T\in \Link{i}{\Delta}} 
\left(\sum_{F\in \Facet{T\cup i}{\Delta}} c_F\right)
\frac{1}{|F|}\frac{|T|!(|F|-|T|-1)!}{(|F|-1)!} (v(T\cup i)-v(T)).
\end{eqnarray*}
Now we can compare the coefficient of $(v(T\cup i)-v(T))$ in the previous equation and the one in the definition of $\operatorname{Shapley}^{\Delta}_i(v)$, see \eqref{eq:Shapley-def-simplicial-complex}, and we assume they are equal:
\[
	\frac{1}{\rank{\Link{i}{\Delta}}}\frac{1}{f_{|T|-1}(\Link{i}{\Delta})}= \left(\sum_{F\in \Facet{T\cup i}{\Delta}} c_F\right)
\frac{1}{|F|}\frac{|T|!(|F|-|T|-1)!}{(|F|-1)!}
\]
and one gets:
\begin{equation}\label{eq:coefficients-comparing-shapley}
	\sum_{F\in \Facet{T\cup i}{\Delta}} c_F 
\frac{\rank{\Link{i}{\Delta}}}{|F|} \frac{f_{|T|-1}(\Link{i}{\Delta})}{{|F|-1 \choose |T|}}=1
\end{equation}

Now, given the simplicial complex, the cardinality of the facets $F$, the rank of each link, the binomial coefficient ${|F|-1 \choose |T|}$ can be treated as constant terms and we denote $$\tilde{c_F}\define c_F 
\frac{\rank{\Star{i}{\Delta}}}{|F| {|F|-1 \choose |T|}}.$$ 

\begin{remark}
	The number $c_F$ is positive if and only if $\tilde{c_F}$ is positive.
\end{remark}


\begin{remark}
	Equation \eqref{eq:coefficients-comparing-shapley} shows two interesting factors:
	\begin{itemize}
		\item $\frac{\rank{\Star{i}{\Delta}}}{|F|}$ describing how far each link is to have the same rank of the full simplex; and
		\item $\frac{\ffi{|T|-1}{\Link{i}{\Delta}}}{{|F|-1 \choose |T|}}$ that describes the discrepancy, level by level, from $\Link{i}{\Delta}$ and a full simplex of the same rank.
	\end{itemize}
\end{remark}

Thus, we have obtained the following equation for every $T\in \Link{i}{\Delta}$: 
\[
	\sum_{F\in \Facet{T\cup i}{\Delta}} \tilde{c_F} f_{|T|-1}(\Link{i}{\Delta})=1.
\]

\begin{theorem}\label{thm:decomposing}
		Let $\Delta$ be a simplicial complex of rank $r$ and let $\mathfrak{I}$ be a cone of cooperative game defined on $\Delta$.
		
		Let $\operatorname{S}^{\Delta}_i(v)$ be individual value written as weighed sum of classical Shapley value, that is, 
		\[
			\operatorname{S}^{\Delta}_i(v)=\sum_{F\in\Facet{i}{\Delta}} c_F \operatorname{Shapley}^{F}_i(v_{|F})
		\]
		where $c_F$ is a real numbers, and where we denote by $\operatorname{Shapley}^{F}_i(v)$ the classical Shapley value for the cooperative game $v_{|F}$ for every facets $F$ of $\Delta$. 

	The Shapley value $\operatorname{Shapley}^{\Delta}_i(v)$ can be written as $S_i^{\Delta}$ if and only if for every $T\in \Link{i}{\Delta}$, there exist a set of real numbers $\{\tilde{c_F}\}_{F\in \Facet{T\cup i}{\Delta}}$ such that 
    \[
	    \sum_{F\in \Facet{T\cup i}{\Delta}} \tilde{c_F} f_{t-1}(\Link{i}{\Delta})=1.
    \]

	The coefficients $c_F$ are defined as
	\[
		c_F = \tilde{c_F} 
\frac{|F| {|F|-1 \choose |T|}}{\operatorname{rank}\Star{i}{\Delta}}.
	\]
\end{theorem}
\begin{proof}
	One implication was already proven in the preparation of the theorem. The converse is shown by reading back the chain of equality.
\end{proof}

\begin{remark}
If the $c_F$'s are treated as probabilities, then we have the further constrains that $\sum_{F\in \Facet{i}{\Delta}} c_F =1$ and $c_F\geq 0$.
\end{remark}

 	\vspace{0.5cm}

\bibliographystyle{amsalpha}
\bibliography{unica}

 	\vspace{0.5cm}
	
 	\noindent
 	{\scshape Ivan Martino}\\
 	{\scshape Department of Mathematics, Royal Institute of Technology.}\\ 
 	{\itshape E-mail address}: \texttt{imartino@kth.se}

\end{document}